\documentclass[review]{elsarticle}
\usepackage{amsthm}
\usepackage{lineno,hyperref}
\usepackage{amssymb,enumerate}
\usepackage{amsmath}
\usepackage{graphics}
\newtheorem{thm}{Theorem}[section]
 
 \newtheorem{lem}[thm]{Lemma}

 \newtheorem{rem}[thm]{Remark}

\journal{Journal of \LaTeX\ Templates}

\begin{document}

\begin{frontmatter}

\title{Life-Span of Semilinear Wave Equations with Scale-invariant Damping: Critical Strauss Exponent Case}

\author{
Ziheng Tu
\footnote{School of Data Science, Zhejiang University of Finance and Economics, 310018, Hangzhou, P.R.China.
 e-mail: tuziheng@zufe.edu.cn.}
\quad
Jiayun Lin
\footnote{
Department of Mathematics and Science, School of Sciences, Zhejiang Sci-Tech University, 310018, Hangzhou, P.R.China.
 email: jylin84@hotmail.com.}
}

\begin{abstract}
The blow up problem of the semilinear scale-invariant damping wave equation with critical Strauss type exponent is investigated. The life span is shown to be: $T(\varepsilon)\leq C\exp(\varepsilon^{-2p(p-1)})$ when $p=p_S(n+\mu)$ for $0<\mu<\frac{n^2+n+2}{n+2}$. This result completes our previous study \cite{Tu-Lin} on the sub-Strauss type exponent $p<p_S(n+\mu)$. Our novelty is to construct the suitable test function from the modified Bessel function. This approach might be also applied to the other type damping wave equations.
\end{abstract}


\end{frontmatter}

\section{Introduction and Main Results}
In this paper, we consider the blow up problem of following semilinear damping wave equation
\begin{equation}\label{main}
\left\{
\begin{array}{ll}
u_{tt}-\Delta u+\frac{\mu}{1+t}u_t=|u|^p\ &(x,t)\in\ \mathbb{R}^n\times[0,\infty) \\
u(x,0)=\varepsilon f(x),\ \ \ \ u_t(x,0)=\varepsilon g(x)\ &x\in\ \mathbb{R}^n,\\
\end{array}
\right.
\end{equation}
where $\mu\geq0$, $f,\ g\in C_0^{\infty}(\mathbb{R}^n)$. We assume that $\varepsilon>0$ is a "small" parameter, and the support of $f$ and $g$ satisfy
\begin{equation}
\mbox{supp}\ (f,\ g)\subset\left\{x\left|\ |x|\leq\frac12\right.\right\}.
\end{equation}
The damping term $\frac{\mu}{1+t}u_t$ is so called as "scale invariant"  since it shares same scaling as $u_{tt}$:
$$\widetilde{u}(t,x)=u(\lambda(1+t)-1,\lambda x).$$
For this typical damping, the asymptotic behavior of linear solution heavily relies on the size of $\mu$ see \cite{Wirt04}. Meanwhile, the blowup problem or the determination of the critical exponent of the semilinear equation has drawn great of attention. Wakasugi \cite{Wakasugi14} has obtained the blowup result if $1<p\leq p_F(n)$ and $\mu>1$, or $1<p\leq 1+\frac2{n+\mu-1}$ and $0<\mu\leq1$. In \cite{WakasugiT}, the upper bound of the lifespan is shown:
\begin{equation*}
\left\{
\begin{array}{ll}
C\varepsilon^{-(p-1)/\{2-n(p-1)\}}\ &\mbox{if}\ 1<p<p_F(n)\ \mbox{and}\ \mu\geq1,\\
C\varepsilon^{-(p-1)/\{2-(n+\mu-1)(p-1)\}}\ &\mbox{if}\ 1<p<1+\frac2{n+\mu-1}\ \mbox{and}\ 0<\mu<1,\\
\end{array}
\right.
\end{equation*}
where $C$ is a positive constant independent of $\varepsilon$ and $p_F(n)=1+\frac2{n}$ is the Fujita exponent. Due to the application of Liouville transform:
$$w(x,t):=(1+t)^{\frac{\mu}2}u(x,t),$$ $w$ turns out to satisfy the wave equation without damping and mass terms when $\mu=2$. There are some studies focusing on this special case. D'Abbicco-Lucente-Reissig \cite{Dabbi15JDE} have showed the small data global existence if:
\begin{equation*}
\left\{
\begin{array}{ll}
2(=p_F(n)=p_S(n+2))<p<\infty\ &\mbox{for}\ n=2,\ \mu=2\\
p_S(5)<p<\infty\  &\mbox{for}\ n=3,\ \mu=2\\
\end{array}
\right.
\end{equation*}
where $p_S(n)$ is the Strauss exponent,
$$p_S(n):=\frac{n+1+\sqrt{n^2+10n-7}}{2(n-1)}$$
which is the positive root of the quadratic equation:
\begin{equation}\label{qua}
\gamma(p,n):=2+(n+1)p-(n-1)p^2=0.\end{equation}
In the case of $1$ dimension and $\mu=2$, Wakasa \cite{Wakasa16} considered the lifespan and showed that the critical exponent is $p_S(3)$. For general $\mu>0$, D'Abbicco \cite{Dabbi15MMAS} obtained following small data global existence result:
\begin{equation*}
\left\{
\begin{array}{ll}
p_F(1)<p<\infty\ &\mbox{for}\ n=1,\ \ \mu\geq\frac53,\\
p_F(2)<p<\infty\ &\mbox{for}\ n=2,\ \ \mu\geq3,\\
p_F(n)<p\leq\frac{n}{n-2}\ &\mbox{for}\ n\geq3,\ \ \mu\geq n+2.\\
\end{array}
\right.
\end{equation*}

Recently, Lai-Takamura-Wakasa \cite{Lai17} applied test function method to obtain following blowup result:
$$T(\varepsilon)\leq C\varepsilon^{-2p(p-1)/\gamma(p,n+2\mu)},$$
$$\mbox{for}\ p_F(n)\leq p< p_S(n+2\mu),\ \ n\geq2\ \ \mbox{and}\ \ 0<\mu<\frac{n^2+n+2}{2(n+2)}.$$
By introducing modified Bessel function $K_{\nu}(t)$ as new test function, Tu-Lin \cite{Tu-Lin} updates their result to:
$$T(\varepsilon)\leq C\varepsilon^{-2p(p-1)/\gamma(p,n+\mu)},$$
$$\mbox{for}\ p_F(n+\frac\mu2)< p<p_S(n+\mu),\ \ n\geq2\ \ \mbox{and}\ \ \mu>0.$$
Ikeda-Sobajima \cite{Ikedapre} obtained similar result. Besides, they also give the life-span of critical case, i.e, $p=p_S(n+\mu)$ with lifespan
$$T(\varepsilon)\leq \exp(C\varepsilon ^{-p(p-1)}),$$
where $\mu$ is required to less than $\mu_*:=\frac{n^2+2+2}{n+2}$.
Their proof relies on the use of hypergeometric function, which is initiated from Zhou-Han \cite{Zhou14}.
In the present paper, we shall give an alternative way to construct the test function.
By applying the wave plane formula and limiting behavior of second type of modified Bessel function,
we succeed to obtain the properties which is critical to the test function. We emphasis that this approach is normal and can be also applied
to some other type of Laplacian with variable coefficients.

Our main result is stated in the following.
\begin{thm}
For the Cauchy problem \eqref{main} with space dimension $n\geq2$, $0<\mu<\mu_*$ and $p=p_S(n+\mu)$, let initial values $f,\ g$ be nonnegative smooth function with compact support in $\{|x|<\frac12\}$. Suppose that a solution $u$ of \eqref{main} satisfies:
$$supp\ u\subset\{(x,t)\in \mathbb{R}^n\times [0,T): |x|\leq t+\frac12\}.$$
Then lifespan $T<\infty$ and there exists a positive constant $C$ which is independent of $\varepsilon$ such that
$$T(\varepsilon)\leq \exp(C\varepsilon^{-p(p-1)}).$$
\end{thm}

\begin{rem}

We note that when $\mu=\mu_*$, the Fujita exponent $p_F(n)$ equals to Strauss exponent $p_S(n+\mu)$. This suggests that $\mu_*$ might be the threshold and the critical exponent of invariant damping wave equation would be in form of $p=\max(p_S(n+\mu),\ p_F(n))$ for any $\mu>0$. However, combing the result of \cite{Dabbi15MMAS} and \cite{Wakasugi14}, there is still gaps. To be precisely, the global small data result of following cases are lacking:
\begin{equation*}
\left\{
\begin{array}{ll}
p_F(1)<p<\infty,\ \frac43<\mu<\frac53,\ &\mbox{for}\ n=1\\
p_F(2)<p<\infty,\ 2<\mu<3,\ &\mbox{for}\ n=2\\
p_F(n)<p\leq\frac{n}{n-2},\ \mu_*(n)<\mu<n+2,\ &\mbox{for}\ n\geq3.\\
\end{array}
\right.
\end{equation*}
and $p>p_S(n+\mu)$ with $0<\mu\leq\mu_*$ for arbitrary $n$. These cases need further consideration.
\end{rem}


The rest of the paper is arranged as follows. In Section $2$, we do some preliminary work. The test function is constructed and some important properties related with are proved. Section $3$ is devoted to prove the main theorem.

\section{Preliminaries}
In this section we prepare some lemmas. We first introduce the second type of modified Bessel function:
$$K_{\nu}(t)=\int_0^\infty\exp(-t\cosh \zeta)\cosh(\nu \zeta)d\zeta,\ \nu\in \mathbb{R},$$
which satisfies:
$$\left(t^2\frac{d^2}{dt^2}+t\frac{d}{dt}-(t^2+\nu^2)\right)K_\nu(t)=0,\ t>0.$$
We collect some facts about $K_{\nu}(t)$ from \cite{Gaunt}. The monotonicity with respect to the order $\nu$:
\begin{equation}\label{mono}
K_\nu(t)<K_{\nu-1}(t)\ \ \mbox{for}\ \nu<\frac12\ \ \ \mbox{and}\ \ \ K_\nu(t)\geq K_{\nu-1}(t)\ \ \mbox{for}\ \nu\geq\frac12.
\end{equation}
The limiting forms of $K_{\nu}(t)$, for $t\gg1$,
\begin{equation}\label{K1}
K_{\nu}(t)=\sqrt{\frac{\pi}{2t}}e^{-t}[1+O(t^{-1})].
\end{equation}
For $0<t\ll1$,
\begin{equation}\label{K2}
K_{\nu}(t)\sim\frac12\Gamma(|\nu|)(\frac12t)^{-|\nu|}\ \ \mbox{for}\ \nu\neq0 \ \ \mbox{and}\ \ K_{\nu}(t)\sim -\ln t\ \ \mbox{for}\ \nu=0.
\end{equation}
The derivative identity:
\begin{eqnarray}\label{K3}
\frac{d}{dt}K_{\nu}(t)&=&-K_{\nu+1}(t)+\frac{\nu}{t}K_{\nu}(t)\\
&=&-\frac12(K_{\nu+1}(t)+K_{\nu-1}(t)).\label{K4}
\end{eqnarray}

We now start to construct our test function $b_q(x,t)$.
For $\eta>0$, let $$\lambda_\eta(t)=(\eta(t+1))^{\frac{\mu+1}{2}}K_{\frac{\mu-1}2}(\eta(t+1))$$
and
$$\varphi_{\eta}(x)=\int_{S^{n-1}}e^{\eta x\cdot \omega}d\omega.$$
It is obvious that $\lambda_\eta(t)$ satisfies
\begin{equation}\label{lambda}
[(1+t)^2\frac{d^2}{dt^2}-\mu(1+t)\frac{d}{dt}+(\mu-\eta^2(1+t)^2)]\lambda_\eta(t)=0, \\
\end{equation}
and $\varphi_\eta(x)$  satisfies
\begin{equation}\label{varphi}\Delta\varphi_{\eta}(x)=\eta^2\varphi_\eta(x).\end{equation}
Define the positive smooth function:
$$b_q(x,t)=\int_0^1\lambda_\eta(t)\varphi_\eta(x)\eta^{q-1}d\eta.$$
Observing \eqref{K2}, to ensure the integrability, $q$ is required to satisfy:
\begin{equation}\label{restri}\frac{\mu+1}2-\frac{|\mu-1|}2+q-1>-1\Longleftrightarrow q>-\min(\mu,1).\end{equation}
We summary the properties of $b_q(x,t)$ in following lemma.
\begin{lem}
$(i)$  $b_{q}(x,t)$ satisfies following two identities $$\left(\frac{\partial^2}{\partial t^2}-\frac\mu{1+t}\frac{\partial}{\partial t}+\frac{\mu}{(1+t)^2}\right)b_{q}(x,t)=b_{q+2}(x,t)$$and
$$\Delta b_q(x,t)=b_{q+2}(x,t).$$
$(ii)$ The following identity holds for $b_{qt}(x,t)$
\begin{equation}\label{bq2}
\frac{\partial}{\partial t}b_q(x,t)=\frac{\mu}{1+t}b_q(x,t)-\int_0^1(\eta(t+1))^{\frac{\mu+1}2}K_{\frac{\mu+1}2}(\eta(t+1))\varphi_\eta(x)\eta^qd\eta.
\end{equation}
$(iii)$ There exists $T_0$, such that for $t>T_0$
\begin{equation}\label{bq3}
b_q(x,t)\sim\left\{
\begin{array}{ll}
(t+1)^{\frac\mu2}(t+1+|x|)^{-q-\frac{\mu}2}\ &\mbox{if}\ -\frac{\mu}2<q<\frac{n-\mu-1}2,\\
(t+1)^{\frac\mu2}(t+1+|x|)^{\frac{-n+1}2}(t+1-|x|)^{\frac{n-\mu-1}{2}-q}\ &\mbox{if}\ q>\frac{n-\mu-1}2.\\
\end{array}
\right.
\end{equation}
\end{lem}
\begin{proof}
$(i)$ The two identities can be proved directly from definition, which means $b_q(x,t)$ satisfy the conjugate equation of \eqref{main}:
$$v_{tt}-\Delta v-(\frac\mu{1+t}v)_t=0.$$
$(ii)$ Utilizing \eqref{K3}, we have
\begin{eqnarray*}
&&\frac{d}{dt}\lambda_\eta(t)=\frac{d}{dt}[(\eta(t+1))^{\frac{\mu+1}{2}}K_{\frac{\mu-1}2}(\eta(t+1))]\\
&=&\frac{\mu+1}2(\eta(t+1))^{\frac{\mu-1}2}\eta K_{\frac{\mu-1}2}(\eta(t+1))+(\eta(t+1))^{\frac{\mu+1}2}\frac{d}{dt}K_{\frac{\mu-1}2}(\eta(t+1))\\
&=&\frac{\mu+1}{2(1+t)}\lambda_\eta(t)+(\eta(t+1))^{\frac{\mu+1}{2}}\eta\left[-K_{\frac{\mu+1}2}(\eta(t+1))+\frac{\mu-1}{2\eta(t+1)}K_{\frac{\mu-1}{2}}(\eta(t+1))\right]\\
&=&\frac{\mu}{1+t}\lambda_\eta(t)-\eta(\eta(t+1))^{\frac{\mu+1}{2}} K_{\frac{\mu+1}2}(\eta(t+1)).
\end{eqnarray*}
Inserting this into $b_q(x,t)$, it is easy to verify \eqref{bq2}.\\
$(iii)$
Applying the plane wave formula to $\varphi_{\eta}(x)$ (see \cite{FJohn}, page 8), we have
\begin{eqnarray*}
b_q(x,t)&=&\int_0^1\lambda_\eta(t)\varphi_{\eta}(x)\eta^{q-1}d\eta\\
&=&\int_0^1\lambda_\eta(t)\omega_{n-1}\int_{-1}^1(1-\theta^2)^{\frac{n-3}{2}}e^{\theta\eta|x|}d\theta\eta^{q-1}d\eta\\
&=&\omega_{n-1}\int_{-1}^1(1-\theta^2)^{\frac{n-3}2}d\theta\int_0^1(t+1)^{\frac{\mu+1}2}K_{\frac{\mu-1}2}(\eta(t+1))e^{\theta\eta|x|}\eta^{q-1+\frac{\mu+1}2}d\eta\\
&=&\omega_{n-1}(t+1)^{-q}\int_{-1}^1(1-\theta^2)^{\frac{n-3}2}d\theta\int_0^{t+1}K_{\frac{\mu-1}2}(\tilde{\eta})e^{\theta\tilde{\eta}\frac{|x|}{t+1}}\tilde{\eta}^{q-1+\frac{\mu+1}2}d\tilde{\eta},
\end{eqnarray*}
where $\tilde{\eta}=(1+t)\eta$.
Due to \eqref{K1}, there exists finite large $T_0$ which is independent with $\varepsilon$ such that for any $\theta\in [-1,1]$ and $(x,t)$ with $\frac{|x|}{1+t}<1$,
\begin{eqnarray*}
\int_{T_0}^{t+1}K_{\frac{\mu-1}2}(\tilde{\eta})e^{\theta\tilde{\eta}\frac{|x|}{t+1}}\tilde{\eta}^{q-1+\frac{\mu+1}2}d\tilde{\eta}&\sim&
\int_{T_0}^{t+1}\sqrt{\frac{\pi}{2\tilde{\eta}}}e^{-\tilde{\eta}+\theta\tilde{\eta}\frac{|x|}{t+1}}\tilde{\eta}^{q-1+\frac{\mu+1}2}d\tilde{\eta}.
\end{eqnarray*}
On the other hand, by \eqref{restri} with \eqref{K2} and the condition $q>-\frac\mu2$, the above integrals on both sides over $[0,T_0]$ are bounded by $C(\mu,\ q,\ T_0)$. Thus we have
$$\int_{0}^{t+1}K_{\frac{\mu-1}2}(\tilde{\eta})e^{\theta\tilde{\eta}\frac{|x|}{t+1}}\tilde{\eta}^{q-1+\frac{\mu+1}2}d\tilde{\eta}\sim\int_{0}^{t+1}\sqrt{\frac{\pi}{2\tilde{\eta}}}e^{-\tilde{\eta}+\theta\tilde{\eta}\frac{|x|}{t+1}}\tilde{\eta}^{q-1+\frac{\mu+1}2}d\tilde{\eta}.$$
Consequently, for $t>T_0$
\begin{eqnarray*}
b_q(x,t)&\sim&(t+1)^{-q}\int_{-1}^1(1-\theta^2)^{\frac{n-3}2}d\theta\int_0^{t+1}e^{-\tilde{\eta}+\theta\tilde{\eta}\frac{|x|}{t+1}}\tilde{\eta}^{q-1+\frac{\mu}2}d\tilde{\eta}\\
&=&(t+1)^{\frac\mu2}\int_{-1}^1(1-\theta^2)^{\frac{n-3}2}(t+1-\theta|x|)^{-(q+\frac{\mu}2)}d\theta\int_0^{t+1-\theta|x|}e^{-\zeta}\zeta^{q-1+\frac\mu2}d\zeta
\end{eqnarray*}
where $\zeta=\frac{t+1-\theta|x|}{t+1}\eta$ and $q>-\frac\mu2$. Notice that for $-1\leq\theta\leq1$, $\frac12\leq t+1-\theta|x|\leq\infty$, we have
$$c_0\leq\int_0^{t+1-\theta|x|} e^{-\zeta}\zeta^{q-1+\mu/2}d\zeta\leq C_0,$$
with $c_0:=\int_0^{\frac12} e^{-\zeta}\zeta^{q-1+\mu/2}d\zeta$ and $C_0:=\int_0^\infty e^{-\zeta}\zeta^{q-1+\mu/2}d\zeta=\Gamma(q+\frac\mu2)$.
Thus
\begin{eqnarray*}
b_q(x,t)&\sim&(t+1)^{\frac\mu2}\int_{-1}^1(1-\theta^2)^{\frac{n-3}2}(t+1-\theta|x|)^{-(q+\frac{\mu}2)}d\theta\\
&=&2^{n-2}(t+1)^{\frac\mu2}\int_{0}^1(\tilde{\theta}(1-\tilde{\theta}))^{\frac{n-3}2}(t+1+|x|-2\tilde{\theta}|x|)^{-(q+\frac{\mu}2)}d\tilde{\theta}\\
&=&2^{n-2}(t+1)^{\frac\mu2}(t+1+|x|)^{-q-\frac{\mu}2}\int_{0}^1(\tilde{\theta}(1-\tilde{\theta}))^{\frac{n-3}2}(1-\tilde{\theta}z)^{-(q+\frac{\mu}2)}d\tilde{\theta}.
\end{eqnarray*}
Here $\tilde{\theta}=\frac{\theta+1}2$ and $z=\frac{2|x|}{t+1+|x|}$. Direct analysis on integral shows for $n\geq 2$,
$$h(z):=\int_{0}^1(\tilde{\theta}(1-\tilde{\theta}))^{\frac{n-3}2}(1-\tilde{\theta}z)^{-(q+\frac{\mu}2)}d\tilde{\theta}$$
is integrable for $0\leq z<1$. For $z=1$, if $-\frac\mu2<q<\frac{n-\mu-1}2$, $h(z)$ is continuous at $1$, thus bounded over $[0,1]$. Otherwise $q>\frac{n-\mu-1}2$,
$h(z)$ behaves as $(1-z)^{\frac{n-\mu-1}2-q}$ around $z=1$. Finally we conclude that
\begin{equation*}
b_q(x,t)\sim\left\{
\begin{array}{ll}
(t+1)^{\frac\mu2}(t+1+|x|)^{-q-\frac{\mu}2}\ &\mbox{if}\ -\frac\mu2<q<\frac{n-\mu-1}2,\\
(t+1)^{\frac\mu2}(t+1+|x|)^{\frac{-n+1}2}(t+1-|x|)^{\frac{n-\mu-1}{2}-q} &\mbox{if}\ q>\frac{n-\mu-1}2.\\
\end{array}
\right.
\end{equation*}
\end{proof}
\begin{rem}
In fact, for $\gamma>\beta>0$, the hypergeometric function has following integral representation
$$F(\alpha,\beta,\gamma;z)=\frac{\Gamma(\gamma)}{\Gamma(\eta)\Gamma(\gamma-\beta)}\int_0^tt^{\beta-1}(1-t)^{\gamma-\beta-1}(1-zt)^{-\alpha}dt,\ |z|<1.$$
That is to say, for $t>T_0$,
\begin{equation*}
b_q(x,t)\sim(t+1)^{\frac\mu2}(t+1+|x|)^{-q-\frac\mu2}F(q+\frac\mu2,\frac{n-1}2,n-1;\frac{2|x|}{t+1+|x|}).
\end{equation*}
\end{rem}
\begin{rem}
As we shall choose $q=\frac{n-\mu-1}{2}-\frac1p$ and $p=p_S(n+\mu)$ in later application, the restriction on $q$, i.e,
$$q>-\min(\mu,1)\ \ \mbox{and}\ \ \ q>-\frac{\mu}2$$
implies that $\mu$ has to satisfy
$$\mu<\mu_*:=\frac{n^2+n+2}{n+2},$$
where $\mu_*$, as we mentioned earlier, also satisfies $p_F(n)=p_S(n+\mu_*)$.
\end{rem}
\begin{lem}
Suppose the Cauchy problem \eqref{main} has an energy solution $u$. If the initial data satisfies that
$$\int_{\mathbb{R}^n}\varphi_1(x)f(x)dx\geq0,\ \ \ \int_{\mathbb{R}^n}\varphi_1(x)g(x)dx\geq0$$
and they are not identically to $0$. Then for $t>T_0$, for any $p>1$
\begin{equation}\label{Priori}
\int_{\mathbb{R}^n}|u(x,t)|^pdx\geq C\varepsilon^p(1+t)^{n-1-\frac{n+\mu-1}2 p}.
\end{equation}
\end{lem}
The proof of this lemma can be found in Tu-Lin \cite{Tu-Lin}. We summaries the key point in Appendix.
We now test equation \eqref{main} by $b_q(x,t)$.
\begin{eqnarray*}
\int_{\mathbb{R}^n}|u|^pb_qdx&=&\int_{\mathbb{R}^n}u_{tt}b_q-u\Delta b_q+\frac{\mu u_t}{1+t}b_qdx\\
&=&\int_{\mathbb{R}^n}u_{tt}b_q-u b_{q+2}+\mu\partial_t[\frac{b_q}{1+t}u]-\mu u\partial_t(\frac{b_q}{1+t})dx\\
&=&\int_{\mathbb{R}^n}u_{tt}b_q+\mu\partial_t[\frac{b_q}{1+t}u]-u\left( b_{q+2}+\frac{\mu b_q'}{1+t}-\frac{\mu b_q}{(1+t)^2}\right)dx\\
&=&\int_{\mathbb{R}^n}u_{tt}b_q+\mu\partial_t[\frac{b_q}{1+t}u]-ub_{qtt}dx\\
&=&\frac{d^2}{dt^2}\int_{\mathbb{R}^n}ub_qdx+\frac{d}{dt}\int_{\mathbb{R}^n}(\mu \frac {b_q}{1+t}u-2b_{qt}u)dx.
\end{eqnarray*}
Integrate this identity over $[0,t]$ three times, we obtain:
\begin{eqnarray*}
&&\frac12\int_0^t(t-\tau)^2\int_{\mathbb{R}^n}b_q|u|^pdxd\tau=\int_0^t\int_{\mathbb{R}^n}ub_qdxd\tau-\varepsilon t\int_{\mathbb{R}^n}fb_q(0)dx\\
&&+\int_0^t(t-\tau)\int_{\mathbb{R}^n}(\frac{\mu b_q}{1+\tau}u-2b_{qt}u)dxd\tau-\frac12\varepsilon t^2\int_{\mathbb{R}^n}gb_q(0)+\mu fb_q(0)-fb_{qt}(0)dx.
\end{eqnarray*}
Inserting \eqref{bq2} for $b_{qt}(0)$, we find the last integral is positive,
\begin{eqnarray*}&&\int_{\mathbb{R}^n}gb_q(0)+\mu fb_q(0)-fb_{qt}(0)dx\\
&&=\int_{\mathbb{R}^n}\left(gb_q(0)+f\int_0^1\eta^{\frac{\mu+1}2}K_{\frac{\mu+1}2}(\eta)\varphi_\eta(x)\eta^{q}d\eta\right) dx>0.
\end{eqnarray*}
We obtain following inequality:
\begin{eqnarray}
&&\frac12\int_0^t(t-\tau)^2\int_{\mathbb{R}^n}b_q|u|^pdxd\tau\nonumber\\
&&\leq\int_0^t\int_{\mathbb{R}^n}ub_qdxd\tau+\int_0^t(t-\tau)\int_{\mathbb{R}^n}(\frac{\mu b_q}{1+\tau}u-2b_{qt}u)dxd\tau.\label{Ineq}
\end{eqnarray}

Based on the estimation \eqref{Ineq}, we have following Lemma.
\begin{lem}
Let $n\geq2$, $p=p_S(n+\mu)$ and $q=\frac{n-\mu-1}2-\frac1p$, under the condition of Theorem 1.1, the functional
\begin{equation}
G(t)=\int_0^t(t-\tau)(1+\tau)\int_{\mathbb{R}^n}b_q(\tau,x)|u(\tau,x)|^pdxd\tau
\end{equation}
satisfies
\begin{equation}\label{G'(t)}
G'(t)\geq K (\ln(1+t))^{1-p}(1+t)\left(\int_0^t(1+\tau)^{-3}G(\tau)d\tau\right)^p\ \ t>2
\end{equation}
where $K$ is a constant which is independent of $\varepsilon$.
\end{lem}
\begin{proof}
We note that the restriction \eqref{restri} on $q$ require $\mu<\mu_*:=\frac{n^2+n+2}{n+2}$.
Simple calculation gives
\begin{eqnarray*}
G'(t)&=&\int_0^t(1+\tau)\int_{\mathbb{R}^n}b_q(\tau,x)|u(\tau,x)|^pdxd\tau,\\
G''(t)&=&(1+t)\int_{\mathbb{R}^n}b_q(t,x)|u(t,x)|^pdx.
\end{eqnarray*}
We estimate the righthand side of \eqref{Ineq}. By using the finite propagation speed and H\"{o}lder inequality, we have
$$I:=\int_0^t\int_{\mathbb{R}^n}ub_qdxd\tau\leq(G'(t))^{\frac1p}\left(\int_0^t\int_{|x|\leq \tau+1}b_q(\tau)(1+\tau)^{-\frac{p'}{p}}dxd\tau\right)^{\frac1{p'}}.$$
Applying \eqref{bq3} for $q=\frac{n-\mu-1}2-\frac1p<\frac{n-\mu-1}2$, such that for $t>T_0$
\begin{eqnarray*}
&&\int_{0}^t\int_{|x|\leq \tau+1}b_q(\tau)(1+\tau)^{-\frac{p'}{p}}dxd\tau\\
&=&\int_{0}^{T_0}\int_{|x|\leq \tau+1}b_q(\tau)(1+\tau)^{-\frac{p'}{p}}dxd\tau+\int_{T_0}^t\int_{|x|\leq \tau+1}b_q(\tau)(1+\tau)^{-\frac{p'}{p}}dxd\tau\\
&\sim&C_1(T_0)+\int_0^t\int_{|x|\leq \tau+1}(\tau+1)^{\frac\mu2}(\tau+1+|x|)^{-q-\frac{\mu}2}(1+\tau)^{-\frac{p'}{p}}dxd\tau\\
&\leq&\int_0^t\int_{|x|\leq \tau+1}(\tau+1)^{-q-\frac{p'}{p}}dxd\tau\\
&\leq&C(T_0)(1+t)^{n-q-\frac{p'}p+1}.
\end{eqnarray*}
Since $q=\frac{n-1-\mu}{2}-\frac1p$ and $p=p_S(n+\mu)$, utilizing the quadratic form \eqref{qua}, we have
\begin{equation}\label{pS}
n-q-\frac{p'}{p}=p'.
\end{equation}
So we obtain
\begin{equation}\label{I}
I\leq C(1+t)^{1+\frac1{p'}}(G'(t))^{\frac1p}.
\end{equation}
For the second term, inserting \eqref{bq2}, we have
\begin{eqnarray*}
&&II:=\int_0^t(t-\tau)\int_{\mathbb{R}^n}(\frac{\mu b_q}{1+\tau}u-2b_{qt}u)dxd\tau
\\
&&=\int_0^t(t-\tau)\int_{\mathbb{R}^n}\left(\frac{-\mu b_q}{1+\tau}\right.\left.+2\int_0^1(\eta(\tau+1))^{\frac{\mu+1}2}K_{\frac{\mu+1}2}(\eta(\tau+1))\varphi_\eta(x)\eta^qd\eta\right) udxd\tau\\
&&=-\mu\int_0^t\frac{t-\tau}{1+\tau}\int_{\mathbb{R}^n}b_qudxd\tau\\
&&+2\int_0^t(t-\tau)\int_{\mathbb{R}^n}\left(\int_0^1(\eta(\tau+1))^{\frac{\mu+1}2}K_{\frac{\mu+1}2}(\eta(\tau+1))\varphi_\eta(x)\eta^qd\eta \right)udxd\tau\\
&&:=II_1+II_2.
\end{eqnarray*}
The estimate of $II_1$ is similar as $I$, by H\"{o}lder inequality and \eqref{bq3},
\begin{eqnarray*}
II_1&=&-\mu\int_0^t\frac{t-\tau}{1+\tau}\int_{\mathbb{R}^n}b_qudxd\tau\\
&\leq&\mu\left(\int_0^t\int_{\mathbb{R}^n}(1+\tau)b_q |u|^pdxd\tau\right)^{\frac1p}\left(\int_0^t\int_{|x|<\tau+1}(\frac{t-\tau}{1+\tau}(1+\tau)^{-\frac1p}b_q^{\frac1{p'}})^{p'}dxd\tau\right)^{\frac1{p'}}\\
&=&\mu(G'(t))^{\frac1{p}}\left(\int_0^t(t-\tau)^{p'}(1+\tau)^{-p'(1+\frac1p)}\int_{|x|<\tau+1}b_q(x,\tau)dxd\tau\right)^{\frac1{p'}}\\
&\leq&\mu(G'(t))^{\frac1{p}}\left(C_2(T_0)+C\int_{T_0}^t(t-\tau)^{p'}(1+\tau)^{-p'(1+\frac1p)+n-q}d\tau\right)^{\frac1{p'}}.
\end{eqnarray*}
Since $q=\frac{n-1-\mu}{2}-\frac1p$ with $p=p_S(n+\mu)$ which implies $-p'(1+\frac1p)+n-q=0$,
$$\int_{T_0}^t(t-\tau)^{p'}(1+\tau)^{-p'(1+\frac1p)+n-q}d\tau=\int_{T_0}^t(t-\tau)^{p'}d\tau\leq(1+t)^{1+p'}.$$
Thus
\begin{equation}\label{II1}
II_1\leq C(\mu,\ T_0)(1+t)^{1+\frac1{p'}}(G'(t))^{\frac1{p}}.
\end{equation}
We now estimate $II_2$. Define
$$\widetilde{b}_{q}(x,t)=\int_0^1(\eta(\tau+1))^{\frac{\mu+1}2}K_{\frac{\mu+1}2}(\eta(\tau+1))\varphi_\eta(x)\eta^{q-1}d\eta,$$
with same process as $(iii)$ in Lemma 2.1, one can prove that $\widetilde{b}_{q}(x,t)$ shares the same estimation \eqref{bq3} of $b_{q}(x,t)$ for some $t>\widetilde{T}_0$.
As $q=\frac{n-1-\mu}2-\frac1p<\frac{n-1-\mu}2<q+1$, for $t>T_1:=\max(T_0,\widetilde{T_0})$,
$$\frac{\widetilde{b}_{q+1}(x,\ \tau)}{b_q(x,\ \tau)}\sim
(1+\tau)^{q+\frac{\mu-n+1}2}(\tau+1-|x|)^{-q-1+\frac{n-\mu-1}2}$$
which combines the H\"{o}lder inequality, it gives
\begin{eqnarray*}
II_2\leq C(T_1)+ (G'(t))^{\frac1p}\left(\int_{T_1}^t(t-\tau)^{p'}\int_{|x|\leq\tau+\frac12}b_q(\frac{\widetilde{b}_{q+1}}{b_q})^{p'}(1+\tau)^{-\frac{p'}{p}}dxd\tau\right)^{\frac1{p'}}.
\end{eqnarray*}
For $\tau>T_1$, due to $q=\frac{n-1-\mu}2-\frac1p$ and \eqref{pS},
\begin{eqnarray*}
&&b_q(\frac{\widetilde{b}_{q+1}}{b_q})^{p'}(1+\tau)^{-\frac{p'}{p}}\\
&\sim&(1+\tau)^{-q+p'(q+\frac{\mu-n+1}2)-\frac{p'}{p}}(\tau+1-|x|)^{p'(-q-1+\frac{n-\mu-1}2)}\\
&=&(1+\tau)^{1-n}(\tau+1-|x|)^{-1}.
\end{eqnarray*}
Thus
\begin{eqnarray*}
&&\int_{T_1}^t(t-\tau)^{p'}\int_{|x|\leq\tau+\frac12}b_q(\frac{\widetilde{b}_{q+1}}{b_q})^{p'}(1+\tau)^{-\frac{p'}{p}}dxd\tau\\
&\sim&\int_{T_1}^t(t-\tau)^{p'}\int_{|x|\leq\tau+\frac12}(1+\tau)^{1-n}(\tau+1-|x|)^{-1}dxd\tau\\
&\leq&\int_{T_1}^t(t-\tau)^{p'}\int_{0}^{\tau+\frac12}(\tau+1-r)^{-1}drd\tau\\
&\leq&\int_0^t(t-\tau)^{p'}\ln(\tau+1)d\tau\leq C(p)(1+t)^{1+p'}\ln(1+t).
\end{eqnarray*}
Consequently, we have
\begin{eqnarray}
II_2&\leq& C(\mu,\ T_0)(1+t)^{1+\frac1{p'}}(G'(t))^{\frac1{p}}\nonumber\\
&+&C(T_1,p)(\ln(1+t))^{\frac1{p'}}(1+t)^{1+\frac1{p'}}(G'(t))^{\frac1p}.\label{II2}
\end{eqnarray}
Plugging the estimates \eqref{I}, \eqref{II1} and \eqref{II2} into \eqref{Ineq}, we obtain for $t>2$,
\begin{eqnarray*}
&&\int_0^t(t-\tau)^2\int_{\mathbb{R}^n}b_q|u|^pdxd\tau\nonumber\\
&\leq& C\left((1+t)^{1+\frac1{p'}}(G'(t))^{\frac1{p}}+(\ln(1+t))^{\frac1{p'}}(1+t)^{1+\frac1{p'}}(G'(t))^{\frac1p}\right)\\
&\leq& C(\ln(1+t))^{\frac1{p'}}(1+t)^{1+\frac1{p'}}(G'(t))^{\frac1p}.
\end{eqnarray*}
On the other hand,
\begin{eqnarray*}
&&\int_0^t(t-\tau)^2\int_{\mathbb{R}^n}b_q|u|^pdxd\tau=\int_0^t(t-\tau)^2(1+\tau)^{-1}G''(\tau)\tau\\
&=&\int_0^t\partial_{t}^2[(t-\tau)^2(1+\tau)^{-1}]G(\tau)d\tau=2(t+1)^2\int_0^t(1+\tau)^{-3}G(\tau)d\tau.\\
\end{eqnarray*}
So we conclude that for $t>2$,
$$G'(t)\geq K (\ln(1+t))^{1-p}(1+t)\left(\int_0^t(1+\tau)^{-3}G(\tau)d\tau\right)^p.$$
\end{proof}

\section{Proof of Main Theorem}
The following lemma is from \cite{Zhou14}.
\begin{lem}
Suppose that $K(t)$ and $h(t)$ are all positive $C^2$ functions and satisfy
$$a(t)K''(t)+K'(t)\geq b(t)K^{1+\alpha}(t),\ \ \ \forall\ t\geq 0,$$
$$a(t)h''(t)+h'(t)\geq b(t)h^{1+\alpha}(t),\ \ \ \forall\ t\geq 0,$$
where $\alpha\geq0$, and
$$a(t),\ b(t)>0\ \ \ \forall\ t\geq 0.$$
If $$K(0)>h(0),\ \ K'(0)\geq h'(0).$$
Then we have
$$K'(t)>h'(t),\ \ \ \forall\ t\geq 0,$$
which implies
 $$K(t)>h(t),\ \ \ \forall\ t\geq 0.$$
\end{lem}
Now we begin to prove the main theorem.
Making use of Lemma 2.3 \eqref{Priori} and Lemma 2.1 \eqref{bq3}, for $t>T_0$,
\begin{eqnarray*}
G(t)&=&\int_0^t(t-\tau)(1+\tau)\int_{\mathbb{R}^n}b_q(\tau,x)|u(\tau,x)|^pdxd\tau\\
&>&\int_{T_0}^t(t-\tau)(1+\tau)\int_{\mathbb{R}^n}b_q(\tau,x)|u(\tau,x)|^pdxd\tau\\
&\geq&\varepsilon^p\int_{T_0}^t(t-\tau)(1+\tau)^{1-q+n-1-\frac{n+\mu-1}2 p}d\tau.
\end{eqnarray*}
Since $q=\frac{n-\mu-1}2-\frac1p$, $p=p_S(n+\mu)$ the exponent
$$1-q+n-1-\frac{n+\mu-1}2p=\frac1p+\frac{n+\mu+1}2-\frac{n+\mu-1}2p=\frac1{2p}\gamma(n+\mu,p)=0.$$
So we have
\begin{equation}\label{25}
G(t)\geq C\varepsilon^p(t-T_0)^2>\frac12 C\varepsilon^pt^2\ \ \ \ \ \mbox{for}\ \ t>4T_0.
\end{equation}
At this moment, we are at exactly same position as Section $3$ of \cite{Zhou14}.
By similar argument, one can obtain the desired result. Here we only give brief outline for completeness.
Define
$$H(t)=\int_0^t(1+\tau)^{-3}G(\tau)d\tau,$$
then
$$H'(t)=(1+t)^{-3}G(t).$$
By \eqref{25}, we have
$$H(t)\geq C\varepsilon\int_0^t(1+\tau)^{-3}\tau^2d\tau\geq c_0\varepsilon^p\ln(1+t)\ \ \mbox{for}\ \ t>T_0,$$
and
$$H'(t)=(1+t)^{-3}G(t)\geq C\varepsilon^p(1+t)^{-3}t^2\geq c_0\varepsilon^p(1+t)^{-1}\ \ \mbox{for}\ \ t>T_0.$$
Plugging $G(t)=(1+t)^3H'(t)$ and $H(t)=\int_0^t(1+\tau)^{-3}G(\tau)d\tau$ into \eqref{G'(t)}, we have
$$((1+t)^3H'(t))'\geq K_0 (\ln(1+t))^{1-p}(1+t)H(t)^p,\ t>2$$
which deduce to
$$(1+t)^2H''(t)+3(1+t)H'(t)>K_0(\ln(1+t))^{1-p}H(t)^p,\ t>2.$$

Set $t+1=\exp(\tau)$ and define
$$H_0(\tau)=H(\exp(\tau)-1)=H(t),$$
one has
$$H_0'(\tau)=H'(t)\frac{dt}{d\tau}=(1+t)H'(t),$$
$$H_0''(\tau)=(1+t)^2H''(t)+(t+1)H'(t).$$
So we obtain following differential inequality system about $H_0(\tau)$, for $\tau>\mathcal{T}_0$
\begin{equation}
\begin{cases}
\begin{aligned}\label{H_0}
&H_0''(\tau)+2H_0'(\tau)>K_0\tau^{1-p}H_0^p(\tau),\\
&H_0(\tau)\geq c_0\varepsilon^p \tau,\\
&H'_0(\tau)\geq c_0\varepsilon^p.
\end{aligned}
\end{cases}
\end{equation}
Let $$H_1(s)=\varepsilon^{p^2-2p}H_0(\varepsilon^{-p(p-1)}s),$$
then we have
\begin{equation}
\begin{cases}
\begin{aligned}\label{H_1}
&\varepsilon^{p(p-1)}H_1''(s)+2H_1'(s)>K_0s^{1-p}H_1^p(s),\\
&H_1(s)\geq c_0s,\\
&H'_1(s)\geq c_0 .
\end{aligned}
\end{cases}
\end{equation}
On the other hand, let $H_3(s)$ satisfy the Ricatti equation
\begin{equation}
\begin{cases}
\begin{aligned}\label{H_3}
&H_3'(s)=\delta H_3^{\frac{p+1}2}(s),\ s\geq s_0,\\
&H_3(s)=c_0/4,
\end{aligned}
\end{cases}
\end{equation}
which blowups at finite time $s^*$ independent with $\varepsilon$. By suitable choosing the parameter $s_0,\ \delta$, so that $c_0\ll s_0\ll\frac1\delta$, one can further construct $H_2(s)=sH_3(s)$ which satisfies
\begin{equation}
\begin{cases}
\begin{aligned}\label{H_2}
&\varepsilon^{p(p-1)}H_2''(s)+2H_2'(s)<K_0s^{1-p}H_2^p(s),\\
&H_2(s)= c_0s_0/4,\\
&H'_2(s)< c_0 .
\end{aligned}
\end{cases}
\end{equation}

Applying Lemma 3.1, to $H_1$ and $H_2$, we have
$$H_1(s)>H_2(s)\ \ s\geq s_0,$$
which implies $H_1$ blows up before $s^*$, i.e, $s<s^*$. The transforming of variables $\tau=\varepsilon^{-p(p-1)}s$ and $t<\exp(\tau)-1$ give the final conclusion.

\section{Appendix}
In this appendix, we give proof of Lemma 2.4.
\begin{proof}
Let
$$\psi(x,t)=\lambda_1(t)\varphi_1(x),$$ and define the functional
$$G_1(t):=\int_{\mathbb{R}^n}u(x,t)\psi(x,t)dx.$$
Then H\"{o}lder inequality gives
\begin{equation}\int_{\mathbb{R}^n}|u(x,t)|^pdx\geq\frac{|G_1(t)|^p}{(\int_{|x|\leq t+1}\psi^{p'}(t,x)dx)^{p-1}}.\label{factor}\end{equation}
Testing \eqref{main} by $\psi(x,t)$, utilizing \eqref{lambda} and \eqref{varphi} for $\eta=1$, the integration by part gives
$$\int_{\mathbb{R}^n}(u_t\psi-u\psi_t+\frac{\mu}{1+s}u\psi) dx\bigg|_0^t=\int_0^t\int_{\mathbb{R}^n}|u|^p\psi dxds>0.$$
Thus we have
\begin{equation}\label{G'1}G_1'(t)+\big(\frac{\mu}{1+t}-2\frac{\lambda_1'(t)}{\lambda_1(t)}\big)G_1(t)\geq\varepsilon C_{f,g},\end{equation}
with $C_{f,g}:=\int_{\mathbb{R}^n}(\lambda_1(0)g(x)+(\mu-2\frac{\lambda_1'(0)}{\lambda_1(0)})f(x))\varphi_1(x)dx$. By the compact support of $f$ and $g$, $C_{f,g}$ is finite. Moreover, $C_{f,g}$ is positive. Since by \eqref{K4}
$$\lambda_1'(t)=\frac{\mu+1}2(1+t)^{\frac{\mu-1}2}K_{\frac{\mu-1}2}(1+t)-\frac12(1+t)^{\frac{\mu+1}2}\left(K_{\frac{\mu+1}2}(1+t)+K_{\frac{\mu-3}2}(1+t)\right),$$
then by \eqref{mono}, we have
\begin{eqnarray*}
\mu-2\frac{\lambda_1'(0)}{\lambda_1(0)}&=&\mu-2\frac{\frac{\mu+1}2K_{\frac{\mu-1}{2}}(1)-\frac12\left(K_{\frac{\mu+1}2}(1)+K_{\frac{\mu-3}{2}}(1)\right)}{K_{\frac{\mu-1}2}(1)}\\
&=&-1+\frac{K_{\frac{\mu+1}2}(1)+K_{\frac{\mu-3}2}(1)}{K_{\frac{\mu-1}2}(1)}>0.
\end{eqnarray*}
Multiplying $\frac{(1+t)^\mu}{\lambda^2(t)}$ on two sides of \eqref{G'1}, then integrating over $[0,t]$, we derive
\begin{equation} G_1(t)\geq\varepsilon C_{f,g}\int_0^t\frac{(1+t)K^2_{\frac{\mu-1}{2}}(1+t)}{(1+s)K^2_{\frac{\mu-1}{2}}(1+s)}ds.\label{G1}\end{equation}
The denominator can be estimated as
\begin{eqnarray}
&&\int_{|x|\leq t+1}\psi^{p'}(t,x)dx\leq \lambda_1^\frac{p}{p-1}(t)\int_{|x|\leq t+1}\varphi_1^{p'}(x)dx\nonumber\\
&&\leq C(1+t)^{n-1+\frac{p(\mu+1)-(n-1)p}{2(p-1)}}e^{\frac{p}{p-1}(t+1)}K^{\frac{p}{p-1}}_{\frac{\mu-1}2}(1+t).\label{deno}
\end{eqnarray}
Combing the estimate \eqref{G1}, \eqref{deno} and \eqref{factor}, we now have
\begin{eqnarray*}&&\int_{\mathbb{R}^n}|u(x,t)|^pdx\\
&&\geq C\varepsilon^p(1+t)^{\frac p2(2-n-\mu)+(n-1)}\\
&& \cdot e^{-p(t+1)}K^p_{\frac{\mu-1}2}(1+t)(\int_0^t\frac1{(1+s)K^2_{\frac{\mu-1}2}(1+s)}ds)^p.
\end{eqnarray*}
Applying \eqref{K1}, for $t>T_0$
\begin{eqnarray*}&&e^{-p(t+1)}K^p_{\frac{\mu-1}2}(1+t)(\int_0^t\frac1{(1+s)K^2_{\frac{\mu-1}2}(1+s)}ds)^p\\
&\geq& C(\frac{\pi}{2(1+t)})^\frac p2 e^{-2p(t+1)}(\int_{\frac t2}^t\frac{2}\pi e^{2(1+s)}ds)^p\\
&&=(\frac{\pi}{2(1+t)})^\frac p2 e^{-2p(t+1)}[\frac1\pi(e^{2(1+t)}-e^{2+t})]^p\geq C(p)(1+t)^{-\frac p2}.
\end{eqnarray*}
Thus
$$\int_{\mathbb{R}^n}|u(x,t)|^pdx\geq C_{1}\varepsilon^p(1+t)^{\frac p2(1-n-\mu)+(n-1)}\ \ \mbox{for}\ \ t>T_0.$$
\end{proof}


.
\newpage
\bibliographystyle{elsarticle-num}
\bibliography{References}

\end{document}